\documentclass{amsart}
\usepackage{amsfonts,amsmath,amsthm,color, amssymb}
\usepackage[all]{xy}
\usepackage[pagebackref]{hyperref}

\newtheorem{theorem}{Theorem}[section]
\newtheorem{lemma}[theorem]{Lemma}
\newtheorem{proposition}[theorem]{Proposition}
\newtheorem{corollary}[theorem]{Corollary}

\newcommand{\Pin}[1]{{\mathbb P}^{#1}}

\newcommand{\rk}[1]{{\rm rk}\,(#1)}

\newcommand{\T}{\mathcal{T}}
\newcommand{\tensor}{\otimes}

\theoremstyle{definition}

\newtheorem{example}[theorem]{Example}

\theoremstyle{remark}
\newtheorem{remark}[theorem]{Remark}

\numberwithin{equation}{section}

\begin{document}

\title{The Varieties of Bifocal Grassmann Tensors}
\author{Marina Bertolini}
\address{{\bf{Marina Bertolini}}. Universit\`a degli Studi di Milano \\
Dipartimento di Matematica ``F. Enriques''\\
Via Cesare Saldini 50\\
20133 Milano \\
E-mail: {\tt marina.bertolini@unimi.it}
}
\author{Gilberto Bini} 
\address{{\bf{Gilberto Bini}}. Universit\`a degli Studi di Palermo \\
Dipartimento di Matematica e Informatica \\
Via Archirafi 34 \\
90123 Palermo \\
E-mail: {\tt gilberto.bini@unipa.it}
}
\author{Cristina Turrini}
\address{{\bf{Cristina Turrini}}. Universit\`a degli Studi di Milano \\
Dipartimento di Matematica ``F. Enriques"\\
Via Cesare Saldini 50\\
20133 Milano \\
E-mail: {\tt cristina.turrini@unimi.it}
}

\date{\today}


\begin{abstract}

Grassmann tensors arise from classical problems of scene reconstruction in computer vision. In particular, bifocal Grassmann tensors, related to a pair of projections from a projective space onto view-spaces of varying dimensions, generalise the classical notion of fundamental matrices. In this paper we study in full generality the variety of bifocal Grassmann tensors focusing on its birational geometry. To carry out this analysis, every object of multi-view geometry is declined both from an algebraic and geometric point of view, e.g., the duality between the view spaces and the space of rays is explicitly described via polarity. Next, we deal with the moduli of bifocal Grassmann tensors, thus showing that this variety is both birational to a suitable homogeneous space and endowed with a dominant rational map to a Grassmannian.
\end{abstract}

\maketitle

\section{Introduction}

Recently, several authors have been interested in the study of some algebraic varieties, which arise within the branch of computer vision called {\em Multiple View Geometry}. In this context, the most investigated varieties are the multiview varieties (see, for example, \cite{L-M}, \cite{A-S-T},\cite{BigL},\cite{I-M-U}), the varieties of trifocal and quadrifocal tensors (\cite{oe1}, \cite{al-to1}, \cite{al-to2}, \cite{Hart-Zi2}, \cite{oe2}) and the critical loci varieties (\cite{be-ber-no-tu}, \cite{tubbLAIA},  \cite{ber-tu-no1}). 

The analysis of the varieties of trifocal and quadrifocal tensors concerns tensors which are defined in the classical case of reconstruction of a three-dimensional static scene from three or four
two-dimensional images. Moreover, they fit in the wide study of Grassmann tensors and their moduli spaces. Grassmann tensors (or multifocal tensors) have been introduced in \cite{Hart-Schaf} as a means of reconstructing a scene in a high dimensional space from its projection by a suitable number of images. More specifically,  they describe the relationships existing between the different images of the same point of the scene taken from different cameras. Moreover, the first and the third author have studied critical loci for projective reconstruction from multiple views, \cite{tubbLAIA},
\cite{be-tur1}, and in this setting Grassmann tensors play a fundamental role \cite{be-ber-no-tu},  \cite{ber-tu-no1}.

In this context, we propose to study in full generality the variety of bifocal Grassmann tensors (or generalized fundamental matrices), which may be viewed as a parameter space of Grassmann tensors of two views from a $k$ dimensional projective space to two image spaces of dimensions $h_1$ and $h_2$, respectively.  In particular, we focus on the birational geometry of this variety. Hence this paper takes into account the behaviour of generic bifocal Grassmann tensors and can be thought of as a first step towards the analisys of the birational geometry of the variety of trifocal Grassmann tensors.

To carry out this analysis, we preliminarily need to decline some basic notions from multiview geometry into a purely algebraic setting. More precisely, computer vision and algebraic geometry are classically linked because taking a picture is described as a linear projection from the ambient space $\Pin{3}$ to a view plane $\Pin{2}$. Additionally, other types of shootings, like videos or segmented scenes, have been more recently interpreted as projections from higher dimensional spaces $\Pin{k}$ to $\Pin{h}$, for suitable $k$ and $h$.

In this setting, a {\it scene} is a set of points $\{X_i\} \in \Pin{k}, i =1, . . . ,N,$ a {\it camera} is a projection from $\Pin{k}$ onto a view space $\Pin{h}$, $(h < k)$, from a linear center. Once homogeneous coordinates have been chosen in $\Pin{k}$ and $\Pin{h}$, the camera can be identified with a $(h + 1) \times (k + 1)$ matrix $P$ of maximal rank, and the center $C_P$ is its right annihilator, hence a $(k-h-1)$-space defined by the linear subspaces of $\Pin{k}$, given by the rows of $P$. These subspaces can also be identified with points of the dual space $(\Pin{k})^{\vee}$ where they span a linear space of dimension $h$. Finally, the right action of $GL(k + 1)$ on $P$ corresponds to a change of coordinates in $\Pin{k}$, while the left action of $GL(h+1)$ can be thought of as a change of coordinates in the view space $\Pin{h}$. 

In the first section of the paper (Section $2$), we frame the above definitions in an algebraic context and we provide the corresponding geometric interpretation of all the involved spaces, i.e. the ambient space, the view space, the space of rays (where a {\it ray} is a fiber of the projection map) and the wedge product spaces of all of them. 
In particular, in the case of one projection, we give an explicit interpretation of the duality between the space of rays and the view space via a polarity correspondence associated with a suitable quadric in $\Pin{k}$, which naturally arises from the projection matrix. Next, we focus on the case of two projections because this is the setting where bifocal Grassmann tensors can be defined, and we describe the action of the natural groups on all these spaces and, again, on their wedge products.

Finally, in the paper \cite{BBBT1}, the authors have computed the rank of bifocal and trifocal Grassmann tensors using a canonical form of the tensors obtained via the actions described above. Here, in the case of bifocal tensors, we use this canonical form in order to give a minimal decomposition of the tensor which has a particular and interesting geometric interpretation.

As a conclusion of this first part, in order to clarify all the previous reasonings, we provide an example for which we perform explicitly all the computations (Example \ref{examplep4p3}).

\
In Sections $3$ and $4$, we deal with bifocal Grassmann tensors and their moduli. Bifocal Grassmann tensors (or generalized fundamental matrices) have been extensively studied in \cite{tubbAMPA}, where their rank is computed and where, in Section $4$  a seminal idea on the structure of their variety is contained. 
Starting from that, in this paper we describe the birational structure of the variety ${\mathcal X}_{(\alpha_1, \alpha_2)}$ of bifocal Grassmann tensors for pairs of projections from $\Pin{k}$ to $\Pin{h_1}$ and to $\Pin{h_2}$ for any admissible choice of $k, h_1, h_2 $ and of a {\it profile} $(\alpha_1,\alpha_2)$ with $\alpha_1+\alpha_2=k+1$, $1 \leq \alpha_i \leq h_i$, $i=1,2$ (\cite{H-K}). 

\vskip 1.5cm

The main results obtained in the paper are the following:

\

\noindent \textbf{Theorem $1$} (see Theorem \ref{homogeneous})
	{\em For each pair $(\alpha_1, \alpha_2)$ corresponding to a profile, the variety of bifocal Grassmann tensors ${\mathcal X}_{(\alpha_1, \alpha_2)}$ is birational to a homogeneous space with respect to the action of $GL(h_1+1) \times GL(h_2+1)$.}

\

\noindent \textbf{Theorem $2$} (see Theorem \ref{ratfibr}) {\em Let $\alpha_1, \alpha_2$ be a pair of non-negative integers such that $\alpha_1+\alpha_2=k+1$. Fix $h_1, h_2$ such that $k > \max \{h_1, h_2\}$ and $k \leq h_1+h_2+1$, as well as a $(k+1)$-dimensional vector space $U$. Set $s_j=h_j+1-\alpha_j$ for $j=1,2$. Then there exists a dominant rational map $\Phi: {\mathcal X}_{(\alpha_1, \alpha_2)} \dashrightarrow G(i, U^{\vee})$ such that the following hold:
\begin{itemize}
    \item $G(i,U^{\vee})$ is birationally ${\mathcal G}$-equivariant, that is, there exists a non-empty open set ${\mathfrak U}$ of ${\mathcal X}_{(\alpha_1,\alpha_2)}$ such that $\Psi(g.p)=\Psi(p)$ for every $p \in {\mathfrak U}$ and every $g \in {\mathcal G}$;
    \item the general orbit is isomorphic to $PGL(i)$,
\end{itemize}	 
where the group ${\mathcal G}$ is the $({\mathbb C^*})^2/{\mathbb C}^*$ quotient of a group isomorphic to  $GL(i) \times GL(h_1+1)  \times GL(h_2+1)$.}
\

Actually, in Theorem \ref{ratfibr} we prove this result for the variety ${\mathcal X}_{(s_1,s_2)}$ which is birational to ${\mathcal X}_{(\alpha_1, \alpha_2)}$, as introduced and discussed before Remark \ref{strat}.
 
Throughout,  we work over the field of complex numbers.

\section{A review on linear projections}

\subsection{Notations}

Let $V$ be a finite dimensional vector space. We denote by ${\mathbb P}(V)$ the projective space of one-dimensional subspaces of $V$. In what follows, $V^{\vee}$ will denote the dual vector space of $V$. Let $F^{\vee} :V_2^{\vee} \rightarrow V_1^{\vee}$ be the transpose map of a linear map $F: V_1 \to V_2$ between finite dimensional vector spaces. If $W$ is a subspace of $V$, the orthogonal space $W^\perp \subseteq V^{\vee}$ consists of all the linear forms on $V$ vanishing on $W$. Then the dual vector space $(V/W)^{\vee}$ is isomorphic to $W^\perp $. This isomorphism sends a linear form $f:V/W \rightarrow \mathbb{C}$ to the linear form $f \circ p_W: V \rightarrow \mathbb{C}$, where $p_W: V \to V/W$ denotes the natural linear projection.

\subsection{The case of one projection} Let $U$ be a $(k+1)$-dimensional vector space. Fix a proper subspace $C \subset U$ of dimension $k-h$ (with $h < k)$, and consider the quotient map $p_C: U \rightarrow U/C$. Notice that $U/C$ can be identified with the $(h+1)-$dimensional space of all the $(k+1-h)$-dimensional subspaces of $U$ containing $C$. Recall the isomorphism $C^\perp \simeq (U/C)^{\vee}$.

\medskip

\subsubsection{Geometric interpretation} Let ${\mathbb P}(U)$ be the projective space associated with $U$, and  $\pi_C: {\mathbb P}(U) \dashrightarrow {\mathbb P}(U/C)$, the  rational map induced by $p_C$, which is well-defined everywhere except on ${\mathbb P}(C)$. As mentioned in the Introduction, in the computer vision setting, we will call $\pi_C$ {\it{camera}} and ${\mathbb P}(C)$ {\it{center}} of the camera $\pi_C$; the target space ${\mathbb P}(U/C)$ is the {\it{space of rays.}} We deduce that a point of ${\mathbb P}(U/C)$ can be identified with a projective linear $(k-h)$-dimensional subspace of ${\mathbb P}(U)$ containing the center ${\mathbb P}(C)$, which will be called a {\it{ray}}. As usual, we will identify ${\mathbb P}(U^\vee)$ with the linear space of hyperplanes of $U$ so that we can identify  ${\mathbb P}((U/C)^{\vee})$ with the subspace of hyperplanes containing ${\mathbb P}(C)$, as $(U/C)^{\vee} \simeq C^{\perp}$. According to the standard setting introduced for the study of algebraic varieties arising in computer vision (see, e.g., \cite{oe1}, \cite{oe2}), we will call 
${\mathbb P}((U/C)^{\vee})$ the {\it{view space}}.

\medskip

In the following, it will be useful to have a model of the target space embedded in ${\mathbb P}(U)$: for this purpose one can choose a projective subspace $L \subset {\mathbb P}(U) $ of dimension $h$, i.e., a {\it screen}, such that $L \cap C = \emptyset$.  Indeed, in this case, the projection map sends a point of ${\mathbb P}(U) \setminus   {\mathbb P}(C)$ to the point of intersection of its ray with $L$.

\subsubsection{The coordinate framework} Fix bases in $U$ and in $U/C$. Then we obtain a representative projection matrix $A$ of size $(h+1) \times (k+1)$ and rank $h+1$ for $h<k$ (defined only up to a non-zero constant). The columns of $A$ generate $U/C$ and the rows of $A$ generate $C^\perp \subset U^\vee$.

\subsection{The case of two projections} Let us choose two proper subspaces $C_1$ and $C_2$ in $U$ such that $\dim(C_1)=k-h_1$, $\dim(C_2)=k-h_2$ and $C_1 \cap C_2=\{0\}$. By Grassmann's Formula, the dimension of the span $C_1+C_2$ is $2k-h_1-h_2=k+1-(h_1+h_2+1-k)$. Thus $C_1+C_2$ has codimension $i:=h_1+h_2-k+1$ in $U$.

Denote by $p_1: U \rightarrow U/C_1$ and $p_2: U \rightarrow U/C_2$ the corresponding projection maps. Let us focus on $p_1: U \rightarrow U/C_1$; a similar statement holds for $p_2$. The image $E^2_1$ of $C_2$  via $p_1$  is the subspace $p_1(C_2)=(C_1+C_2)/C_1$ in $U/C_1$, which is isomorphic to $C_2$, as $C_2 \cap C_1=\{0\}$. Let us consider the projection with center $E^2_1,$ $$p^2_1 : U/C_1 \rightarrow  (U/C_1)/ ((C_1+C_2)/C_1) \simeq  U/(C_1+C_2)$$ 
and its composition with $p_1$, namely
\begin{equation}
\label{namely}
U \overset{p_1}{\longrightarrow} U/C_1 \overset{p^2_1}{\longrightarrow} U/(C_1+C_2).
\end{equation}

Analogously, with obvious meaning of the symbols, we have
$$
U \overset{p_2}{\longrightarrow} U/C_2 \overset{p^1_2}{\longrightarrow} U/(C_1+C_2).
$$
Since $p_1$ and $p_2$ are the projections onto $U/C_1$ and $U/C_2$, respectively, and $p^2_1$, $p^1_2$ are induced by $p_1$ and $p_2$, we have the following commutative diagram:
\begin{equation}
\label{primodiagramma}
\xymatrix{
	U \ar[d]_{p_2} \ar[r]^{p_1} & U/C_1 \ar[d]^{p^2_1} \\
	U/C_2 \ar[r]_{\! \! \! \! \! \!p^1_2} & U/(C_1+C_2).
}
\end{equation}

In the dual setting, the vector space $U^{\vee}$ will contain the subspaces $C_1^{\perp}$ and $C_2^{\perp}$ of dimension $h_1+1$ and $h_2+1$, which are isomorphic to $(U/C_1)^{\vee}$ and $(U/C_2)^{\vee}$ respectively. 
Since $(C_1+C_2)^{\perp} = C_1^{\perp} \cap C_2^{\perp}$, we have
\begin{equation}
\label{capandcup}
\left( U/C_1\right )^\vee \cap \left( U/C_2\right)^\vee = \left(U/\left(C_1 + C_2\right) \right)^\vee.
\end{equation}

As a consequence of Grassmann's formula, we get
\begin{eqnarray*}
	\dim \left( \left( U/C_1\right)^\vee \cap \left(U/C_2 \right)^\vee\right)&=& \dim\left( \left(U/C_1 \right)^\vee\right)+\dim\left(\left( U/C_2\right)^{\vee}\right) \\&-&\dim \left( \left(U/C_1\right)^{\vee} + \left(U/C_2 \right)^\vee\right)=i.
\end{eqnarray*}

By dualizing Diagram \ref{primodiagramma}, we have
\begin{equation}
\label{diagramma}
\xymatrix{
	(U/(C_1+C_2))^{\vee} \simeq (U/(C_1))^{\vee} \cap ((U/C_2))^{\vee}   \ar[r]^{{\,\,\,\,\,\,\,\,\,\,\,\,\,\,\,\,\,\,\,\,\,\,\,\,\,\,\,\,\,\,\,\,\,\,\,\,\,\,\,\,\,\,\,\,\,\,\,\,\,\,\,\,\,\,  p^2_1}^{\vee}} \ar[d]^{{p^1_2}^{\vee}} \ar[r] & (U/C_1)^{\vee} \ar[d]^{p^{\vee}_1} \\
	(U/C_2)^{\vee} \ar[r]^{p^{\vee}_2} & U^{\vee}.
}
\end{equation}
In other words, $\left(U/\left(C_1 + C_2\right) \right)^{\vee}$ is the fiber product of $p_1^{\vee}:(U/C_1)^{\vee} \rightarrow U^{\vee}$ and $p_2^{\vee}:(U/C_2)^{\vee} \rightarrow U^{\vee}$. 

\begin{lemma}
\label{lemmanosnake}
Assume $C_1 \cap C_2=\{0\}$. The vector space $\left(U/\left(C_1 + C_2\right) \right)^{\vee}$ is isomorphic to $ker(\eta^{\vee})$, where
\begin{equation}
    \label{etadefinition}
\eta:=p_1 \oplus (-p_2): U \longrightarrow U/C_1 \oplus U/C_2.
\end{equation}
\end{lemma}
\begin{proof}

Since $C_1 \cap C_2=\{0\}$, $\eta$ is injective and the following exact sequence holds:
$$
0 \rightarrow U \rightarrow U/C_1 \oplus U/C_2 \rightarrow coker(\eta) \rightarrow 0.
$$
If we dualize the short exact sequence above, we have
$$
0 \rightarrow ker(\eta^{\vee}) \rightarrow (U/C_1)^{\vee} \oplus (U/C_2)^{\vee} \rightarrow U^{\vee} \rightarrow 0,
$$
where $\eta^{\vee}=p_1^{\vee}\oplus(- p_2^{\vee}).$ We construct an explicit isomorphism between  $coker(\eta)$ and $U/(C_1 + C_2)$, so that the thesis will follow by duality.  

It is easy to check that an isomorphism 
 $$ \phi: (U/C_1 \oplus U/C_2)/\eta(U) \rightarrow U/(C_1 + C_2)$$ can be defined as follows:
 $$\phi([([a]_1,[b]_2)]_\eta) = [a+b]_{1,2},$$
 where $a,b \in U$, and where $[-]_1,[-]_2,[-]_\eta, [-]_{1,2}$ denote the equivalence classes modulo $C_1, C_2, \eta(U), C_1 + C_2$, respectively.

\end{proof}
\subsubsection{Geometric interpretation}

Let  $\pi_j: {\mathbb P}(U) \dashrightarrow {\mathbb P}(U/C_j)$, the map induced by $p_j$ onto the target space of rays.
From the assumptions on the centers $C_1$ and $C_2$ we have $\mathbb{P}(C_1) \cap \mathbb{P}(C_2)=\emptyset$. We can view ${\mathbb P}(U/(C_1+C_2))$ as the set of rays through the linear span of ${\mathbb P}(C_1)$ and ${\mathbb P}(C_2)$; denote by $\pi^j_{12}: {\mathbb P}(U/C_j) \dashrightarrow {\mathbb P}(U/(C_1+C_2))$ the natural projections, $j=1,2$. Finally, Diagram \ref{primodiagramma} allows us to define $\pi_{12}: \mathbb{P}(U) \dashrightarrow \mathbb{P}(U/(C_1+C_2))$, as $\pi_{12} = \pi^1_{12} \circ \pi_1 =  \pi^2_{12} \circ \pi_2$. As it is standard in computer vision, we call {\it{epipole}} the projective linear space  $\mathbb{P}(E^i_j) = \pi_j(\mathbb{P}(C_i)) \subseteq \mathbb{P}(U/C_j)$. The epipole $\mathbb{P}(E^i_j)$ can be viewed as the center of the projection $\pi^j_{12}$ and can be identified with $\mathbb{P}((C_1+C_2)/C_j),$ $j=1,2.$

As before, one could also choose, for $j=1,2$, projective subspaces $L_j \subset {\mathbb P}(U) $ of dimension $h_j$ such that $L_j \cap C_j = \emptyset$ as screens,  i.e. models of the view spaces embedded in ${\mathbb P}(U) $.  If the screens are in general position, their intersection $L_1 \cap L_2$ is a projective subspace of dimension $i-1$, where $i:=h_1+h_2-k+1$ and one can also interpret the composition $ \pi^1_{12} \circ \pi_1 = \pi^2_{12} \circ \pi_2 $ as the projection of ${\mathbb P}(U)$ onto the intersection $ L_1 \cap L_2$ of the screens. We can also interpret some subspaces in the dual setting: as we said above ${\mathbb P}((U/C_j)^{\vee})$ is the subspace of hyperplanes containing ${\mathbb P}(C_j)$ and similarly ${\mathbb P}((U/(C_1+C_2)^{\vee}) = {\mathbb P}((U/C_1)^{\vee}) \cap {\mathbb P}((U/C_2)^{\vee}) $ is the subspace of hyperplanes containing ${\mathbb P}(C_1)$ and ${\mathbb P}(C_2)$.

Finally, we recall the definition of corresponding rays and corresponding subspaces coming from the setting of Computer Vision. Let  $R_1 \in \mathbb{P}(U/C_1), R_2 \in \mathbb{P}(U/C_2)$ be a pair of rays. We say that $R_1$ and $R_2$ are {\it{corresponding rays}} if their intersection is not empty, as subspaces of $\mathbb{P}(U)$. Let $ \Lambda_j $ be a general linear subspace of $\mathbb{P}(U/C_j)$ of codimension $\alpha_j$, $j=1,2$. We say that $\Lambda_1$ and $\Lambda_2$ are {\it{ corresponding subspaces}} if their intersection is not empty, as subspaces of $\mathbb{P}(U)$.

\begin{example}
	\label{example_P3}
For $k=4$ and $h_1=h_2=2$, we have two linear projections in ${\mathbb P}^3$ from two distinct points ${\mathbb P}(C_1)$ and ${\mathbb P}(C_2)$ onto two distinct planes, which intersect along a line, as $i=2$ in this case. The map $\pi_{12}$ is the linear projection from the line connecting the two points ${\mathbb P}(C_1)$ and ${\mathbb P}(C_2)$. Moreover, the maps $\pi^1_{12}$ and $\pi^2_{12}$ are projections from the epipoles onto the line of intersections of the screens embedded in $3$-dimensional projective space.
\end{example}

\subsubsection{The coordinate framework}\label{twoproj}

Assume we have two projections $\pi_j: {\mathbb P}(U) \dashrightarrow {\mathbb P}(U/C_j)$ for $j=1,2$ and consider the maps 
$\pi^j_{12}: {\mathbb P}(U/C_j) \dashrightarrow {\mathbb P}(U/(C_1+C_2))$ for $j=1,2,$ where $\pi_{12}$ ($= \pi^1_{12} \circ \pi_1 =  \pi^2_{12} \circ \pi_2$ ) is introduced before.

Fix bases $\mathcal{B}$, $\mathcal{B}_1$, $\mathcal{B}_2$ and $\mathcal{B}_{12}$, for $U, U/C_1, U/C_2$ and $U/(C_1+C_2)$ respectively.
Denote by $A$ (resp. $B$) the full rank $(h_1+1) \times  (k+1)$ (resp. $(h_2+1) \times  (k+1)$) representative matrix of $\pi_1$ (resp. $\pi_2$)  with respect to $\mathcal{B} $ and $\mathcal{B}_1$ (resp. $\mathcal{B}$ and $\mathcal{B}_2$). Also, consider full rank representative matrices $P, N_1$ and $N_2$ of $\pi_{12},  \pi^1_{12}$ and $ \pi^2_{12} $ respectively, with the bases chosen above.
By construction, we have $P=N_1A$ and $P=N_2B$.


In what follows, we need to make a natural choice of the bases in order to have a very simple form for the two matrices $A \in Mat(h_1+1, k+1)$ and $B \in Mat(h_2+1, k+1)$ of maximal rank, which canonically represent the projections $\pi_1$ and $\pi_2$. For these purposes, we pick a basis ${\mathcal C}_1:=\{a_1, \ldots, a_{k-h_1}\}$ of $C_1$ and a basis ${\mathcal C}_2:=\{b_1, \ldots, b_{k-h_2}\}$ of $C_2$.
Since $C_1$ and $C_2$ have zero intersection, the union of these two bases give a basis ${\mathcal C}$ of the sum $C_1+C_2$. Complete ${\mathcal C}$ to a basis ${\mathcal B}:=\{ u_1, \ldots, u_i ,  a_1, \ldots, a_{k-h_1},  b_1, \ldots, b_{k-h_2}\}$ of $U$, where $u_j \notin C_1+C_2$. As for $U/C_1$, we choose the basis ${\mathcal B}_1:=\{[u_1]_1, \ldots,[u_i]_1, [b_1]_{1}, \ldots, [b_{k-h_2}]_{1}\}$,  where $[-]_1$ denotes the equivalence class modulo $C_1$. Analogously for $U/C_2$, we choose the basis ${\mathcal B}_2:=\{[u_1]_2, \ldots,[u_i]_2, [a_1]_{2}, \ldots, [a_{k-h_2}]_{2}\}$, where $[-]_2$ denotes the equivalence class modulo $C_2$. With this choice, the matrices associated with $\pi_1$ and $\pi_2$ are given by
$$
\tilde{A}=\left(
\begin{array}{ccc}
I_i & 0_{i,k-h_2} & 0_{i,k-h_1} \\
0_{k-h_2, i} & I_{k-h_2} & 0_{k-h_2
	, k-h_1}
\end{array}
\right),
$$
$$
\tilde{B}=\left(
\begin{array}{ccc}
I_i & 0_{i,k-h_2} & 0_{i,k-h_1} \\
0_{k-h_1, i} & 0_{k-h_1
	, k-h_2} & I_{k-h_1}
\end{array}
\right),
$$
where $I_t$ is the $t \times t$ identity matrix and $0_{a,b}$ is the zero matrix with $a$ rows and $b$ columns. By definition, the epipole $E^2_1$ in $U/C_1$ (resp. the epipole $E^1_2$ in $U/C_2$) is generated by the vectors $[b_1]_{1}, \ldots, [b_{k-h_2}]_{1}$ (resp. $[a_1]_{2}, \ldots, [a_{k-h_1}]_{2}$). The matrix associated with $\pi^1_{12}$ has $i$ rows and $h_1+1$ columns; the matrix associated with $\pi^2_{12}$ has $i$ rows and $h_2+1$ columns. If we choose the bases ${\mathcal B}_1$, ${\mathcal B}_2$ and $\mathcal{B}_{12} =\{[u_1]_{12}, \ldots, [u_i]_{12}\}$, where $[-]_{12}$ denotes the equivalence classes modulo $C_1+C_2$, the matrices corresponding to $\pi^1_{12}, \pi^2_{12}$ and $\pi_{12}$ are given by
\begin{equation}
    \label{tauci}
\tilde{N}_1=(I_i \, \, 0_{i,k-h_2} ), \qquad \tilde{N}_2=( I_i \, \,0_{i,k-h_1}), \qquad \tilde{P}=(I_i \, \, 0_{i,k+1-i} ).
\end{equation}


\subsection{Polarity with respect to the quadric $A^TA$}
\label{polarity}

Two symmetric matrices are naturally associated with a projection matrix $A$, that is, the matrix $AA^T$ of size $h+1$ and the matrix $A^TA$ of size $k+1$. Both have rank $h+1$ so the former defines a non-singular quadric in the ray space ${\mathbb P}(U/C)$; the latter quadric $Q_A$ lies in ${\mathbb P}(U)$ and has vertex the center of the camera ${\mathbb P}(C)$. The polarity defined by the quadric $Q_A$ induces an explicit isomorphism $\psi_A$ between $\mathbb{P}(U/C)$ and $\mathbb{P}(C^{\perp})$, which associates a ray with its polar hyperplane with respect to the quadric $Q_A$, which passes through the vertex ${\mathbb P}(C)$. If we fix a basis in $U$, thus introducing homogeneous coordinates $[X]$ in projective space ${\mathbb P}(U)$, the quadric $Q_A$ is the set of points $[X] \in {\mathbb P}(U)$ such that $X^TA^TAX=0$. Thus, setting $\psi_A([AX])$ the hyperplane with dual coordinates $A^TAX$, we get a well defined bijective map. As recalled before, the projective space $\mathbb{P}(C^{\perp})$ is isomorphic to ${\mathbb P}((U/C)^{\vee})$. Therefore, the polarity with respect to $Q_A$ gives a canonical map between a ray and the corresponding polar hyperplane. Thus, we give an explicit geometric interpretation of the isomorphism between the ray space and the view space, and we describe - from a more explicit viewpoint - the map associated with a projection matrix introduced by A. Aholt and L. Oeding \cite{oe1}, \cite{oe2}.

\

In the case of two projection matrices, we deal with $3$ quadrics, $Q_A, Q_B$ and $Q_P$ in ${\mathbb P}(U)$. They correspond to the symmetric matrices $A^TA$, $B^TB$ and $P^TP$, respectively. The quadrics $Q_A$ and $Q_B$ are quadric cones with vertices ${\mathbb P}(C_1)$ and ${\mathbb P}(C_2)$; the vertex of the quadric $Q_P$ is the span of the centers ${\mathbb P}(C_1)$ and ${\mathbb P}(C_2)$. Up to projective transformations in ${\mathbb P}(U)$, we can choose $P$ to be an $i \times (k+1)$ given as $P=(T|0)$, where $T$ is an $i \times i $ invertible matrix and $0$ is the zero matrix with $i$ rows and $k+1-i$ columns. The intersection of $Q_A$ (resp. $Q_B$) with the projection screen $L_1$ (resp. $L_2$) is a non-singular quadric $\Gamma_A$ (resp. $\Gamma_B$). Generically, the two screens intersect along an $(i-1)$-dimensional space $L_{12}$, which can be taken as the screen of the projection with associated matrix $P$. The intersection of $Q_P$ with $L_{12}$ is a rank $i$ quadric $Q_{12}$ in $L_{12}$; hence it is non-singular if $i \geq 3$ (the case $i=2$ is shown below in a specific example). As mentioned before, $Q_{12}$ can also be obtained as the quadric associated with the projection of $\Gamma_A$ (resp. $\Gamma_B$) onto $L_{12}$ from $\mathbb{P}(E^2_1)$ (resp. $\mathbb{P}(E^1_2)$).

\
\begin{example}
Let us go back to Example \ref{example_P3}. The quadrics $Q_A$ and $Q_B$ are two cones with vertices the centers of projections. Without loss of generality, assume $C_1=(0:0:0:1)$ and $C_2=(0:0:1:0)$. Up to projective transformations in ${\mathbb P}^3$, we can assume $A=(I_3|0)$, where $0$ is a $3 \times 1$ zero column. The matrix $B$ can be written as $(M|n)$ where $M$ is a $3 \times 3$ matrix and $n$ is a $3 \times 1$ column vector with entries $n_{14}, n_{24}, n_{34}$. Moreover, the third column of $M$ has to be the zero column because of the choice of $C_2$. In this case a natural, not unique, choice of the matrix $P$ is the $2 \times 4$ matrix given by $(T|0_2)$ where $T$ is a $2 \times 2$ invertible matrix and $0_2$ is the $2 \times 2$ matrix of zeros. As a consequence, the equations of $Q_A$ and $Q_B$ are $x_0^2+x_1^2+x_2^2=0$ and $X^TB^TBX=0$, where $[X]$ are homogeneous coordinates in ${\mathbb P}^3$. 

The intersection of $Q_A$ (resp. $Q_B$) with the screen of projections is a non-singular conic. In the case of $C_1$, we can choose $x_3=0$ as a projection screen, so the image of $Q_A$ is the conic $x_0^2+x_1^2+x_2^2=0$, which is non-singular in the plane $x_3=0$.  In the case of $C_2$, we can choose $x_2=0$ as a projection screen, and the image of $Q_B$ is the non singular conic $X^TB^TBX=0, x_2=0$. The epipole $\mathbb{P}(E_1^2)$ is the point $(0:0:1:0)$ while the epipole $\mathbb{P}(E_2^1)$ is the point $(n_{14}: n_{24}: 0: n_{34})$. The line ${\mathbb P}(C_1+C_2)$ has equation $x_0=x_1=0$ and the line $l$ of equation $x_2=x_3=0$ can be chosen as a screen for the projection from ${\mathbb P}(C_1+C_2)$. The projection of the conic $x_0^2+x_1^2+x_2^2=0, x_3=0$ from $\mathbb{P}(E_1^2)$ onto the line $l$ gives two points $V_1^1$ and $V_1^2$. For a generic choice of $n_{14}, n_{24}, n_{34}$ the projection of the conic $X^TB^TBX=0, x_2=0$ from the epipole ${\mathbb P}(E_2^1)$ gives two points $U_2^1$ and $U_2^2$ on $l$.  The pairs of points $V_1^1, V_1^2$, and $U_2^1,U_2^2$ are the same. Indeed the quadric with vertex ${\mathbb P}(C_1+C_2)$ is given by $(t_{11}^2+t_{21}^2)x_0^2+(t_{12}^2+t_{22}^2)x_1^2 + (t_{11}t_{12}+t_{21}t_{22})x_0x_1=0$, where $T=(t_{ij})$ is the matrix above. It has two irreducible components that are planes through ${\mathbb P}(C_1+C_2)$. Generically, the two components intersect the line $x_2=x_3=0$ in two sets of distinct points, $\{V_1^1, V_1^2\}$ and  $\{U_2^1,U_2^2\}$, which coincide due to the commutativity of Diagram \ref{primodiagramma}.

\end{example}

\subsection{A group action on the space of rays and the space of views}
\label{actiongroups}
Coming back to the case of one projection, the general linear group $GL(k+1)$ acts on $U$ on the left. Precisely, pick a basis ${\mathcal B}$ in $U$, any $(k+1) \times (k+1)$ invertible matrix $M$ induces an automorphism $L_M$ of $U$ such that a vector $u \in U$ is mapped to $Mu$. Let us consider the stabilizer ${\mathcal S}_C$ of $C$ in $GL(k+1)$. Fix the basis ${\mathcal B}:=\{a_1, \ldots, a_{k-h},  u_1, \ldots, u_{h+1}\}$ in $U$, which is obtained by fixing a basis ${\mathcal C}:=\{a_1, \ldots, a_{k-h}\}$ of $C$ and completing it to a basis of $U$. Then a matrix of ${\mathcal S}_C$ is a block matrix of the following form:
$$
\left(
\begin{array}{cc}
D_1 & T \\
0 & D_2
\end{array}
\right)
$$
where $D_1 \in GL(k-h)$ and $D_2 \in GL(h+1)$. Let us consider $U/C$, with the induced basis ${\mathcal B'}:=\{[u_1], \ldots,[u_{h+1}]\}$  where, as in the previous sections, $[-]$ denotes the equivalence class modulo $C$. If $M \in {\mathcal S}_C$, there exists a commutative diagram
$$
\begin{array}{ccc}
U & \overset{M}{\rightarrow}  & U \\
A \, \downarrow & & \downarrow \, A\\
U/C & \overset{N_M}{\rightarrow} & U/C
\end{array}
$$
such that $AM=N_MA$. As remarked above, the rows of $A$ are linearly independent, so there exists a pseudo-inverse $A^{\dagger}$ such that $AA^{\dagger}=I$, where $I$ is the identity matrix of size $(h+1)$. As a consequence, we can take $N_M$ as $AMA^{\dagger}$. 

Therefore, the stabilizer ${\mathcal S}_C$ induces a left action on $U/C$. Indeed, for $[r] \in U/C$ there exists $u \in U$ such that $[r]=[Au]$. Then $N_M([r])=(AMA^{\dagger})([r]):=[A(Mu)]$. It is an exercise to verify that this action is well defined. Accordingly, the left action of $PGL(k+1)$ on ${\mathbb P}(U)$ induces a left action of the image of ${\mathcal S}_C$ in $PGL(k+1)$ on the space of rays ${\mathbb P}(U/C)$.

Now, let us start from $U/C$, with the basis fixed before. A matrix $N \in GL(h+1)$ acts on the left on $U/C$. Since a linear map preserves the zero vector, there exists a matrix $M_N \in {\mathcal S}_C$ such that the following diagram commutes:
$$
\begin{array}{ccc}
U & \overset{M_N}{\rightarrow}  & U \\
A \, \downarrow & & \downarrow \, A\\
U/C & \overset{N}{\rightarrow} & U/C
\end{array}
$$
where $M_N=A^{\dagger}NA$ is a matrix in ${\mathcal S}_C$. Therefore we have $N([r])=N([Au])=[A(M_Nu)]$ for $r$ and $u$ such that $[Au]=[r]$. If we consider the transpose maps of the diagram above, we get the natural actions induced by ${M_N}^{T}$ on the dual space $U^{\vee}$ and by ${N}^{T}$ on the space of views $(U/C)^{\vee}$, where ${A}^{T} {M_N}^{T} ={N}^{T} {A}^{T}.$

Finally, any matrix $N \in GL(h+1)$ inducing a linear transformation on the space of rays $U/C$, yields a transformation on the wedge spaces $\bigwedge^j(U/C)$ and $\bigwedge^j(U/C)^{\vee}$: the former is given by the matrix $\Lambda^j N$ and the latter is given by $\bigwedge^jN^T$.

\section{Bifocal Grassmann tensors} \label{grasstens}
 We recall here the basic elements of the
construction of {\it Grassmann tensors} (\cite{Hart-Schaf}), in the case of our interest, i.e. for two projections.

\
Let us consider a pair of projections
$\pi_{j}:  {\mathbb P}(U) \dashrightarrow {\mathbb P}(U/C_j)$ for $j=1,2$, fix a {\it profile}  $(\alpha_1, \alpha_2)$ and choose bases for $U$ and $U/C_j$. Let $\{\mathcal{S}_j\}$ for $j=1,2,$ where $\mathcal{S}_j
\subset {\mathbb P}(U/C_j)$ be a set of general $s_j$-dimensional spaces, with
$s_j=h_j-\alpha_j$. Let $S_j$ be the matrix of size $(h_j+1)\times (s_j+1)$ of maximal rank
 whose columns are a basis for
$\mathcal{S}_j$. By definition, if all the $\mathcal{S}_j$ are
corresponding subspaces there exists a point $\mathbf{X} \in
{\mathbb P}(U)$ such that $\pi_{j}(\mathbf{X})\in \mathcal{S}_j$ for
$j=1,2.$ In other words, there exist $2$ vectors
$\mathbf{v_j} \in \mathbb{C}^{s_j+1}$ $j = 1,2,$ such that
\begin{equation}
\label{grasssystem}
\begin{bmatrix}
P_1 & S_1 & 0   \\
P_2 & 0 & S_2 \\
\end{bmatrix}%
\cdot
\begin{bmatrix}
\mathbf{X}\\
\mathbf{v_1} \\
\mathbf{v_2} \\
\end{bmatrix}
=
\begin{bmatrix}
0 \\
0 \\
\end{bmatrix}.
\end{equation}

The existence of a non-trivial solution
$\{\mathbf{X},\mathbf{v_1},\mathbf{v_2}\}$ of the linear system
(\ref{grasssystem}) implies that the system matrix has zero
determinant. This determinant can be thought of as a bilinear
form, i.e. a tensor, in the Pl\"{u}cker coordinates of the spaces
$\mathcal{S}_j.$ This tensor is called the {\it bifocal Grassmann tensor} $\T,$ and
$\T \in V_1 \tensor V_2$ where $V_j = \bigwedge^{s_j+1}(U/C_j)$ is the
$\binom{h_j+1}{s_j + 1}$-dimensional vector space such that $G(s_j+1, h_j+1) \subset \mathbb{P}(V_j).$
More explicitly, the entries of the Grassmann tensor are some of
the Pl\"{u}cker coordinates of a point in the Grassmannian $G(k+1,U/C_1 \oplus U/C_2 )$, i.e. of the matrix
\begin{equation}
\left[
\begin{array}{c|c}
\label{trasposta_r}
A^T & B^T \\
\end{array}
\right],
\end{equation}
up to sign. More specifically, they are the maximal minors of the matrix
(\ref{trasposta_r}) obtained by selecting $\alpha_1$ columns
from the matrix $A^T$ and $\alpha_2$ columns from the matrix $B^T$.

\begin{remark}In what follows, we give a more abstract description of Grassmann tensors. For these purposes, recall first the Hodge operator. Let $V$ be an $n$-dimensional vector space. Pick $\{b_1, \ldots, b_n\}$ a basis of $V$ such that $1 \in {\mathbb C}$ corresponds to the vector $b_1 \wedge \ldots \wedge b_n \in \bigwedge^n V \simeq {\mathbb C}$. Recall that the Hodge operator is a linear map $*: \bigwedge^k V \to \bigwedge^{n-k} V$ defined as follows. Let $I:=\{i_1 < \ldots  < i_k\}$ be a multi-index and denote by $J:=\{j_1 < \ldots <j_{n-k}\}$ the complementary multi-index in $\{1, \ldots, n\}$. Then we have $*(b_I):=(-1)^{\sigma(I,J)}b_J$, where $b_I:=b_{i_1} \wedge \ldots \wedge b_{i_k}$, where $\sigma(I,J)$ is $+1$ or $-1$ according to the parity of the permutation $(I,J).$

\end{remark}
The subspaces ${\mathcal S}_j$ in $(U/C_j)$ may be viewed as elements of the wedge powers of the direct sum $(U/C_1) \oplus (U/C_2)$. Therefore for any profile $(\alpha_1, \alpha_2)$ we have
\begin{equation}
\bigwedge^{k+1}\left( (U/C_1)\oplus (U/C_2)\right)= \bigoplus_{\alpha_1, \alpha_2}\left( \bigwedge^{\alpha_1} \left( U/C_1\right)\otimes \bigwedge^{\alpha_2}\left(U/C_2\right)\right).
\label{sommawedge}
\end{equation}

Moreover, by the isomorphisms induced by the Hodge operator, we have
\begin{eqnarray}
\label{fundwedge}
\bigwedge^{\alpha_1} \left( U/C_1\right)\otimes \bigwedge^{\alpha_2}\left(U/C_2\right) &\simeq& \bigwedge^{s_1+1} \left(U/C_1\right)^{\vee} \otimes \bigwedge^{s_2+1}\left(U/C_2\right)^{\vee} \\ &=& Hom\left(\bigwedge^{s_1+1}\left(U/C_1\right), \bigwedge^{s_2+1}\left(U/C_2\right)^{\vee} \right).
\end{eqnarray}

Therefore, any Grassmann tensor can be viewed as a linear map, thus yielding a matrix ${\mathfrak F}$ which is called a {\em generalized fundamental matrix} of size $\binom{h_2+1}{h_2-\alpha_2+1} \times \binom{h_1+1}{h_1-\alpha_1+1}$. The entries of ${\mathfrak F}$ can be described explicitly. Let $I=\{i_1 < \dots < i_{s_1+1}\},$ $J=\{j_1 < \dots < j_{s_2+1}\}$ be two multi-indices in $\{1, \ldots, h_1+1\}$ abd $\{1, \ldots, h_2+1\}$, respectively. Denote by $I^c, J^c$ the (ordered) sets of
complementary indices.  Moreover,
denote by $A_I$ and $B_J$ the matrices obtained from $A^T$ and
$B^T$ by deleting the columns corresponding to the indices $i_1, \dots, i_{s_1+1}$ and $j_1, \dots,
j_{s_2+1},$ respectively. Then the entries of $\mathfrak{F}$ are given by
$F_{I,J}=\epsilon(I,J) \det
\begin{bmatrix}
A_I &  B_J \\
\end{bmatrix}$ where $\epsilon(I,J)$ is $+1$ or $-1$ according to the parity of
the permutation $(I,J, I^c,J^c),$ with lexicographical
order of the multi-indices $\{I\}$ for the rows and $\{J\}$
for the columns. In \cite{tubbAMPA} and \cite{BBBT1}, the authors proved the following result:
\begin{theorem}
\label{genfundmatteo} Let us consider two projections of maximal rank with profile ($\alpha_1$, $\alpha_2$). Moreover, assume the intersection of the centers is empty. Then the rank of the corresponding bifocal Grassmann tensor
$\mathfrak{F}$ is given by 
$$\rk{\mathfrak{F}} =\binom{(h_1-\alpha_1+1)+(h_2-\alpha_2+1)}{h_1-\alpha_1+1}.$$
\end{theorem}

\subsection{An action on the set of projection matrices} 
\label{action}
In what follows, fix a $(k+1)$-dimensional vector space $U$. Define ${\mathfrak P}$ to be the vector space of all pairs of matrices $(M_1,M_2)$, where $M_i$ is a matrix of size $(k+1) \times (h_j+1)$ for $j=1,2$. It contains an open set ${\mathfrak W}$ of pairs of matrices $(M_1, M_2)$ such that $M_j$ has maximal rank $h_j+1$. Naturally, it can be identified with an open set in
$$
{\mathbb A}^{(k+1)(h_1+h_2+2)} \simeq {\mathbb A}^{(k+1)(h_1+1)} \times {\mathbb A}^{(k+1)(h_2+1)} \simeq {\mathfrak P}.
$$

\begin{lemma}
\label{voidintersection}
Assume $k \geq h_j+1, j=1,2$ and $k \leq h_1+h_2+1$. The matrix $[M_1|M_2]$ of size $(k+1) \times (h_1+h_2+2)$ has rank $k+1$ if and only if $C_1 \cap C_2= \{0\}$, where $C_j$ is the null-space $Ker(M^T_j)$ for $j=1,2$.
\end{lemma}

\begin{proof} Choose $M_1$ and $M_2$ as above so $\dim(C_j)=k-h_j$. The matrix $[M_1|M_2]$ has rank $k+1$ if and only if its nullspace $N$ has dimension $i:=h_1+h_2+1-k$. On the other hand, $N$ is isomorphic to 
\begin{eqnarray*}
Im(M_1) \cap Im(M_2) &\simeq& Ker(M_1^T)^{\perp}\cap Ker(M_2^T)^{\perp} \\ &\simeq& \left(Ker(M_1^T)+ Ker(M_2^T) \right)^{\perp} \\ &=& (C_1+C_2)^{\perp} \simeq \left( U/(C_1+C_2)\right)^{\vee}.
\end{eqnarray*}

Therefore the matrix $[M_1|M_2]$ has rank $k+1$ if and only if $i= k+1-\dim(C_1+C_2)$, i.e., $\dim(C_1 \cap C_2)=0$ by Grassmann's formula.
\end{proof}

\begin{remark}
\label{hart1}
Let $A$ and $B$ two projection matrices of size $(h_j+1) \times (k+1)$ for $j=1,2$. If we set $M_1=A^T$ and $M_2=B^T$, the matrix $[M_1|M_2]$ gives a point in $G_1=G(k+1,h_1+h_2+2)$. By choosing suitable bases in $U$, $U/C_1$ and $U/C_2$, the matrices $M_1$ and $M_2$ correspond to the linear maps $p_1$ and $p_2$ in Diagram \ref{primodiagramma}.
\end{remark}

Now, let ${\mathfrak W}_0 \subseteq {\mathfrak W}$ be the subset of matrices $[M_1|M_2]$ such that $C_1 \cap C_2 =\{0\}$. There is a left action of $GL(k+1)$ on ${\mathfrak W}_0$, as well as a right action of $GL(h_1+1) \times GL(h_2+1)$ on ${\mathfrak W}_0$, namely:
$$
\begin{array}{ccc}
GL(k+1) \times {\mathfrak W}_0 \times (GL(h_1+1) \times GL(h_2+1))
 &\longrightarrow & {\mathfrak W}_0 \\ & & \\
\left(G,\left[M_1|M_2\right],\left[%
\begin{array}{c|c}
H_1&  \mathbf{0} \\
\hline 
 \mathbf{0} & H_2 \\
\end{array}%
\right] \right) & \longrightarrow & \left[G\,M_1\,H_1 | G\,M_2\,H_2\right],
\end{array}
$$
where ${\mathbf 0}$ is the zero matrix. Let us describe this action more explicitly. For $j=1,2$ let $L_{M_j} = <{M_j}^1, \dots , {M_j}^{h_j+1}>$ be the vector space of dimension $h_j+1$, which is spanned by the columns of $M_j$. Moreover, set $\Lambda_{M_j}=\mathbb{P}(L_{M_j})$. Then, with the same notation as before, the dimension of $I_{M_1,M_2}:= L_{M_1} \cap L_{M_2}$ is equal to $i=h_1+h_2-k+1 > 0$. Moreover, we choose bases
$\{v_1, \dots, v_i, w_{i+1}, \dots, w_{h_1+1}\}$  for  $L_{M_1}$ and 
$\{v_1, \dots, v_i, w'_{i+1}, \dots, w'_{h_2+1}\}$ for  $L_{M_2}$
such that $\{v_1, \dots, v_i\}$ is a basis for $I_{M_1,M_2}$. As a consequence, there exist matrices $K_1 \in GL(h_1+1)$ and $K_2 \in GL(h_2+1)$ such that 
\begin{equation}
\left[%
\begin{array}{c|c}
\label{precanonica}
M_1 & M_2 \\
\end{array}%
\right] \left[%
\begin{array}{c|c}
K_1 &  \mathbf{0} \\
\hline 
 \mathbf{0} & K_2 \\
\end{array}%
\right] = 
\end{equation}
$$
=\left[%
\begin{array}{c|c}
v_1, \dots, v_i, w_{i+1}, \dots, w_{h_1+1} & v_1, \dots, v_i, w'_{i+1}, \dots, w'_{h_2+1} \\
\end{array}
\right].
$$

Under our assumptions, $\{v_1, \dots, v_i, w_{i+1}, \dots,
w_{h_1+1}, w'_{i+1}, \dots, w'_{h_2+1}\}$ is a basis of
$U^{\vee}$, so there exists $G \in PGL(k+1)$ such that
$$G \left[%
\begin{array}{c}
v_1, \dots, v_i, w_{i+1}, \dots, w_{h_1+1} ,  w'_{i+1}, \dots, w'_{h_2+1} 
\end{array}%
\right] = \left[%
\begin{array}{c}
e_1, \dots, e_{k+1} 
\end{array}%
\right],$$ 
where $\{e_1, \dots,
e_{k+1}\}$ is the canonical basis of $\check{\mathbb{C}}^{k+1}$. This implies that
$$G \left[%
\begin{array}{c|c}
\label{canonica}
v_1, \dots, v_i, w_{i+1}, \dots, w_{h_1+1} & v_1, \dots, v_i, w'_{i+1}, \dots, w'_{h_2+1} \\
\end{array}%
\right] = $$
\begin{equation}
\label{canonicalform}
\left[%
\begin{array}{cc|cc}
 I_i & \mathbf{0} & I_i & \mathbf{0} \\
  \mathbf{0} & I_{h_1+1-i} & \mathbf{0} & \mathbf{0} \\
  \mathbf{0} & \mathbf{0} & \mathbf{0} & I_{h_2+1-i} \\
\end{array}%
\right],
\end{equation}
where $I_a$ denotes the $a \times a$ identity matrix. The matrix in \eqref{canonicalform} is called  {\em the canonical form for matrices $[M_1|M_2] \in {\mathfrak W}_0$}. 

Finally, if we look at the Grassmanniann $G_1=G(k+1,h_1+h_2+2)$ as a quotient by the action of $PGL(k+1)$ of rank $(k+1)$ matrices of size $(k+1) \times (h_1+h_2+2)$, the image of ${\mathfrak W}_0$ under the corresponding quotient map is an open subset, which is denoted by ${\mathfrak Q}_0$. By duality, the Grassmanniann $G_1=G(k+1,h_1+h_2+2)$ is isomorphic to the Grassmanniann $G_2=G(i, h_1+h_2+2)$. As a consequence, the image ${\mathfrak G}_0$ of ${\mathfrak Q}_0$ under this isomorphism is an open set in $G_2$. An element of it can be described by means of a matrix $[\tau_1|\tau_2]^T$ where $\tau_j^T$ has size $(h_j+1) \times i$ and maximal rank $i$.  Thus, there is an action of the group $GL(h_1+1) \times GL(h_2+1) \times GL(i)$ on set ${\mathfrak Q}_0$ of such matrices $[\tau_1|\tau_2]^T$, namely:
\begin{equation}
\begin{array}{ccc}
GL(h_1+1) \times GL(h_2+1) \times {\mathfrak Q}_0 \times GL(i)
 &\longrightarrow & {\mathfrak Q}_0 \\ & & \\
\left( \left[\begin{array}{c|c}
\Delta_1&  \mathbf{0} \\
\hline 
 \mathbf{0} & \Delta_2 \\
\end{array}%
\right], 
\begin{array}{c}
[\tau_1|\tau_2]^T
\end{array} 
, \Gamma \right)
 & \longrightarrow & 
 \begin{array}{c}
\left[(\Delta_1 \tau_1^T \Gamma), (\Delta_2 \tau_2^T \Gamma) \right]^T
 \end{array},
\end{array}
\label{tauaction}
\end{equation}
where ${\mathbf 0}$ is the zero matrix.

\begin{remark}
As explained in \eqref{fundwedge}, after choosing suitable bases, any fundamental matrix ${\mathfrak F}$ can be viewed as the matrix associated with a linear map from $\bigwedge^{\alpha_1}  U/C_1$ to $\bigwedge^{\alpha_2} (U/C_2)^{\vee}$, which are isomorphic to $\bigwedge^{s_1+1}(U/C_1)^{\vee}$ and $\bigwedge^{s_2+1}(U/C_2)$ via the Hodge isomorphism. Therefore ${\mathfrak F}$ is related to the fundamental matrix associated with the matrix $[\tau_1|\tau_2]^T$, where the dual Pl\"ucker coordinates appear.
\end{remark}

\subsection{Decomposition of a bifocal Grassmann tensor as sum of indecomposable tensors}

Here we explicitly describe a minimal - not necessarily unique - decomposition of the generalized fundamental matrix ${\mathfrak F}$ as the sum of $rank(\mathfrak F)$ indecomposable tensors (for the different definitions of rank see, for instance, \cite{LA}). For these purposes, we describe the action on the set of generalized fundamental matrices, which is induced by that in the previous section. 

Denote by ${\mathfrak F}_c$ the generalized fundamental matrix associated with the canonical form \eqref{canonicalform}. As recalled in Section \ref{actiongroups}, the connection between the bifocal Grassmann tensor $\mathfrak{F}$ associated with $[M_1|M_2]$ and the bifocal Grassmann tensor $\tilde{\mathfrak{F}}$ arising from (\ref{precanonica}) is given by
\begin{equation}
\label{arising}
\tilde{\mathfrak{F}} = (\bigwedge^{s_2+1}K_2^{-1}) \cdot \mathfrak{F} \bigwedge^{s_1+1}(K_1^{-1})^T.
\end{equation} Moreover, since $G \in GL(k+1)$, we have $\mathfrak{F}_c = det(G) \tilde{\mathfrak{F}}$. In other words, the fundamental matrix associated with $[M_1|M_2] \in {\mathfrak W}_0$ is related to ${\mathfrak F}_c$ as follows:
\begin{equation}
    \label{fundcan}
{\mathfrak F}= (det(G))^{-1}\left( \bigwedge^{s_2+1} K_2 \right) \, {\mathfrak F}_c \,
\left( \bigwedge^{s_1+1} K_1^T \right),
\end{equation}
 where $G$, $K_1$ and $K_2$ are introduced in Section \ref{action}.  

Now, fix bases in $W$, $(U/C_1)^{\vee}$, $(U/C_2)^{\vee}$ where $W=(U/C_1)^{\vee} \cap (U/C_2)^{\vee}$. Then the matrix $\tau_{j}^T$ induces a linear map from $W$ to $(U/C_j)^{\vee}$ for $j=1,2$; hence $\tau_j^T$ is a $(h_j+1) \times i$ matrix. Recall that the Hodge operator $*$ induces an isomorphism between $\bigwedge^{s_1+1} W^{\vee}$ and $\bigwedge^{s_2+1} W$, as the dimension of $W$ is $i$ and $s_1+2+s_2=i$. Also, we have the following commutative diagram, namely:
$$
\xymatrix{
\bigwedge^{s_1+1} W^{\vee} \ar[r]^{*} & \bigwedge^{s_2+1} W \ar[d]^{\bigwedge^{s_2+1} \tau_{2,c}}  \\ 
\bigwedge^{s_1+1}(U/C_1) \ar[r]_{{\mathfrak F}} \ar[u]^{\bigwedge^{s_1+1}\tau^T_{1,c}}& \bigwedge^{s_2+1}(U/C_2)^{\vee},
}
$$
where ${\mathfrak F}$ is by definition a bifocal Grassmann tensor.  Let $I$ be a multi-index of length $s_1+1$ in $\{1, \ldots, i\}$ and denote by $I^{c}$ its complement of length $i-(s_1+1)=s_2+1$. As $I$ varies, denote by $E_I$ the basis of $\bigwedge^{s_1+1}W^{\vee}$ induced by a fixed basis of $W$. Set $F_{I^c}=*(E_I) \in \bigwedge^{s_1+1}W$. 

\begin{proposition}
	\label{decomp_canon_prop} Let 
	$$
	A_c=\left(
	\begin{array}{cc}
	I_{h_1+1} & 0_{h_1+1, k-h_1}  
	\end{array}
	\right), \qquad  \qquad 
	B_c=\left(
	\begin{array}{ccc}
	I_i & 0_{i,k-h_2} & 0_{i,k-h_1} \\
	0_{k-h_1, i} & 0_{k-h_1
		, k-h_2} & I_{k-h_1}
	\end{array}
	\right)
	$$
	be matrices such that $[A_c^T|B_c^T]$ is a $(k+1) \times (h_1+h_2+2)$ matrix, as introduced in Lemma \ref{voidintersection}. Then the corresponding bifocal Grassmann tensor $\mathfrak{F}_c$ has the following minimal decomposition up to sign: 
\begin{equation}
\label{decomp_canon}
\mathfrak{F}_c = \sum_I \left(\left(\bigwedge^{s_1+1} \tau_{1,c} \right) E_{I}\right)   \otimes \left(\left(\bigwedge^{s_2+1}\tau_{2,c}\right) F_{I^c}\right).
\end{equation}
\end{proposition}

\begin{proof} Take the basis $E_I$ in $\bigwedge^{s_1+1} W^{\vee}$ as above. The Hodge operator corresponds - up to sign - to the tensor $\sum_I E_I \otimes *(E_I)= \sum_I E_I \otimes F_{I^c} \in \bigwedge^{s_1+1} W \otimes \bigwedge^{s_2+1} W$. If we apply $\bigwedge^{s_1+1} \tau_{1,c} \otimes \bigwedge^{s_2+1} \tau_{2,c}$ to $\sum_I E_I \otimes *(E_I)$, we have an element in $\bigwedge^{s_1+1}(U/C_1)^{\vee} \otimes \bigwedge^{s_2+1} (U/C_2)^{\vee}$, namely ${\mathcal F}_c$. Thus we have 
\begin{equation*}
{\mathfrak F}_c = \sum_I \left(\left(\bigwedge^{s_1+1} \tau_{1,c}\right) E_I \right) \otimes \left( \left(\bigwedge^{s_2+1} \tau_{2,c} \right) F_{I^c} \right).
\end{equation*}
\end{proof}

\begin{remark}
The sum in (\ref{decomp_canon}) has $\binom{i}{s_1 +1} = \rk{\mathfrak{F}_c}$ addenda, so that (\ref{decomp_canon}) is a minimal decomposition of $\mathfrak{F}_c$ as sum of rank $1$ tensors. Notice that this decomposition may not be necessarily unique.
\end{remark}

The combination of \eqref{fundcan} and \eqref{decomp_canon} allows us to prove the following result. 
\begin{corollary}
\label{effetau}
Let $[\tau_{1,c}|\tau_{2,c}]^T$ be the $(h_1+h_2+2) \times i$ matrix corresponding to ${\mathfrak F}_c$. With the same notation adopted in this section, the following holds (up to sign):
\begin{eqnarray*}
\mathfrak{F} &=& \frac{1}{det(G)}\sum_I \left(\left(\bigwedge^{s_1+1}K_1 \bigwedge ^{s_1+1}\tau_{1,c} \right)E_I \right) \otimes \left( \left(\bigwedge^{s_2+1}K_2\bigwedge^{s_2+1} \tau_{2,c}\right)F_{I^c}\right) \\ &=& \frac{1}{det(G)} \sum_I P_I \otimes Q_{I^c},
\end{eqnarray*}
\end{corollary}
where $P_I \in \bigwedge^{s_1+1}(U/C_1)$ and $Q_{I^c} \in \bigwedge^{s_2+1}(U/C_2)^{\vee}$.

\begin{example}
\label{examplep4p3} Set $(\alpha_1, \alpha_2) = (3,3)$, so $k=5$. Moreover, set $h_1=4$ and $h_2=3$. Consider two projections from $\Pin{5}$ to
$\Pin{4}$ and $\Pin{3}$ with profile $(3,3)$. In this case $i=3$. Pick the matrix $[A^T|B^T]$ of size $6 \times 9$ where $A$ and $B$ are the projection matrices, namely:
$$
\left[
\begin{array}{cccccccccc}
1 & 0 & 0 & 0 & 0 & | & 0 & 0 & 0 & 1 \\
0 & 1 & 0 & 0 & 0 & | & 0 & 0 & 1 & 0 \\
0 & 0 & 1 & 0 & 0 & | & 0 & 1 & 0 & 0 \\
0 & 0 & 0 & 1 & 0 & | & 1 & 0 & 0 & 0 \\
0 & 0 & 0 & 0 & 1 & | & 1 & 0 & 1 & 0 \\
1 & 1 & 1 & 1 & 1 & | & 0 & 1 & 0 & 1 
\end{array}
\right]
$$

Set 
$$
K_1=
\left[
\begin{array}{ccccc}
1 & 0 & 0 & 0 & 0 \\
0 & 1 & 0 & 0 & 0 \\
0 & 0 & 1 & 0 & 0\\
0 & -1 & 0 & 1 & 0 \\
0 & 0 & 0 & 0 & 1
\end{array}
\right],
\qquad
K_2=
\left[
\begin{array}{ccccc}
0 & -1 & 0 & 1 \\
0 & 0 & 1 & 0 \\
0 & 1 & 0 & 0 \\
1 & 0 & 0 & 0
\end{array}
\right].
$$
We have 
\begin{equation*}
\left[%
\begin{array}{c|c}
A^T & B^T \\
\end{array}%
\right] \left[%
\begin{array}{c|c}
K_1 &  \mathbf{0} \\
\hline 
\mathbf{0} & K_2 \\
\end{array}%
\right] = 
\left[
\begin{array}{ccccccccc}
\begin{array}{cccccccccc}
1 & 0 & 0 & 0 & 0 & | & 1 & 0 & 0 & 0 \\
0 & 1 & 0 & 0 & 0 & | & 0 & 1 & 0 & 0 \\
0 & 0 & 1 & 0 & 0 & | & 0 & 0 & 1 & 0 \\
0 & -1 & 0 & 1 & 0 & | & 0 & -1 & 0 & 1 \\
0 & 0 & 0 & 0 & 1 & | & 0 & 0 & 0 & 1 \\
1 & 0 & 1 & 1 & 1 & | & 1 & 0 & 1  & 0
\end{array}
\end{array}
\right],
\end{equation*}
i.e. we have turned the matrix into the form \ref{precanonica}.
Finally we consider the matrix 
$$
G=
\left[
\begin{array}{cccccc}
1 & 0 & 0 & 0 & 0 & 0 \\
0 & 1 & 0 & 0 & 0 & 0 \\
0 & 0 & 1 & 0 & 0 & 0 \\
1/2 & 1/2 & -1/2 & 1/2 & -1/2 & 1/2 \\
-1/2 & -1/2 & -1/2 & -1/2 & 1/2 & 1/2 \\
1/2 & 1/2 & 1/2 & 1/2 & 1/2 & -1/2
\end{array}
\right]
$$
and get
\begin{equation*}
G\left[%
\begin{array}{c|c}
A^T & B^T \\
\end{array}%
\right] \left[%
\begin{array}{c|c}
K_1 &  \mathbf{0} \\
\hline 
 \mathbf{0} & K_2 \\
\end{array}%
\right] = 
\left[
\begin{array}{ccccccccc}
\begin{array}{cccccccccc}
1 & 0 & 0 & 0 & 0 & | & 1 & 0 & 0 & 0 \\
0 & 1 & 0 & 0 & 0 & | & 0 & 1 & 0 & 0 \\
0 & 0 & 1 & 0 & 0 & | & 0 & 0 & 1 & 0 \\
0 & 0 & 0 & 1 & 0 & | & 0 & 0 & 0 & 0 \\
0 & 0 & 0 & 0 & 1 & | & 0 & 0 & 0 & 0 \\
0 & 0 & 0 & 0 & 0 & | & 0 & 0 & 0  & 1
\end{array}
\end{array}
\right],
\end{equation*}
which is the canonical form of $[A^T|B^T]$. Notice that $det(G) = -\frac{1}{2}$. 

The $5 \times 3$ matrix $\tau_1^T$ and the $4 \times 3$ matrix $\tau_2^T$ are given by:
$$
\tau_1^T= \left[
\begin{array}{ccc}
1 & 0 & 0 \\
0 & 1 & 0 \\
0 & 0 & 1 \\
0 & -1 & 0 \\
0 & 0 & 0
\end{array}
\right], \qquad
\tau_2^T= \left[
\begin{array}{ccc}
0 & -1 & 0 \\
0 & 0 & 1 \\
0 & 1 & 0 \\
1 & 0 & 0
\end{array}
\right].
$$


\noindent The transpose of the generalized fundamental matrix ${\mathfrak F}_c$ of the canonical form  above is given by
$$
\left[
\begin{array}{cccc}
0 & 0 & -1 & 0\\
0 & 1 & 0 & 0 \\
0 & 0 & 0 & 0 \\
0 & 0 & 0 & 0 \\
-1 & 0 & 0 & 0 \\
0 & 0 & 0 & 0 \\
0 & 0 & 0 & 0 \\
0 & 0 & 0 & 0 \\
0 & 0 & 0 & 0 \\
0 & 0 & 0 & 0
\end{array}
\right]
$$
which can be decomposed as follows:
\begin{eqnarray*}
&-&\left[
\begin{array}{cccccccccc}
1 & 0 & 0 & 0 & 0 & 0 & 0 & 0 & 0 & 0
\end{array}
\right]
\otimes 
\left[
\begin{array}{cccc}
0 & 0 & 1 & 0
\end{array}
\right] \\ &+& \left[\begin{array}{cccccccccc}
0 & 1 & 0 & 0 & 0 & 0 & 0 & 0 & 0 & 0
\end{array}
\right]
\otimes 
\left[
\begin{array}{cccc}
0 & 1 & 0 & 0
\end{array}
\right] \\ &-& \left[\begin{array}{cccccccccc}
0 & 0 & 0 & 0 & 1 & 0 & 0 & 0 & 0 & 0
\end{array}
\right]
\otimes 
\left[
\begin{array}{cccc}
1 & 0 & 0 & 0
\end{array}
\right].
\end{eqnarray*}

Notice that $(det(G)^{-1}){\mathfrak F_c}^T= \left(\bigwedge^2K_1^{-1}\right) {\mathfrak F}^T (K_2^{-1})^T$ where 
$$
{\mathfrak F}^T=\left[
\begin{array}{cccc}
0 & 2 & 0 & 0\\
2 & 0 & -2 & 0 \\
0 & -2 & 0 & 0 \\
0 & 0 & 0 & 0 \\
0 & 0 & 0 & 2 \\
0 & 0 & 0 & 0 \\
0 & 0 & 0 & 0 \\
0 & 0 & 0 & 2 \\
0 & 0 & 0 & 0 \\
0 & 0 & 0 & 0
\end{array}
\right]
$$
which, up to the constant $det(G)^{-1},$ can be decomposed as follows:
\begin{eqnarray*}
&-&\left[
\begin{array}{cccccccccc}
1 & 0 & -1 & 0 & 0 & 0 & 0 & 0 & 0 & 0
\end{array}
\right]
\otimes 
\left[
\begin{array}{cccc}
0 & 1 & 0 & 0
\end{array}
\right] \\ &+& 
\left[\begin{array}{cccccccccc}
0 & 1 & 0 & 0 & 0 & 0 & 0 & 0 & 0 & 0
\end{array}
\right]
\otimes 
\left[
\begin{array}{cccc}
-1 & 0 & 1 & 0
\end{array}
\right] \\ &-& \left[\begin{array}{cccccccccc}
0 & 0 & 0 & 0 & 1 & 0 & 0 & 1 & 0 & 0
\end{array}
\right]
\otimes 
\left[
\begin{array}{cccc}
0 & 0 & 0 & 1
\end{array}
\right],
\end{eqnarray*}
as predicted in Corollary \ref{effetau}. In particular, $P_I \in \bigwedge^2(U/C_1)$ and $Q_{I^c} \in (U/C_2)^{\vee}$ for every choice of multi-indices $I$.
\end{example}

\section{Moduli spaces of Bifocal Grassmann Tensors}

\subsection{The varieties of generalized fundamental matrices} Fix a vector space $U$ of dimension $k+1$. Assume $\alpha_1$ and $\alpha_2$ are two positive integers such that $\alpha_1+\alpha_2=k+1$. Let $[M_1|M_2]$ be a general point in ${\mathfrak W}_0$. Recall that $M_1^T$ and $M_2^T$ are two general projection matrices, in the sense of Lemma \ref{voidintersection}. Notice that $C_j=Ker(M_j^T)$ for $j=1,2$. Therefore, we have a linear projection 
$$
\pi: {\mathbb P}\left(\bigwedge^{k+1}(U/C_1 \oplus U/C_2) \right) \dashrightarrow {\mathbb P}\left(\bigwedge^{\alpha_1}(U/C_1) \otimes \bigwedge^{\alpha_2} (U/C_2) \right) 
$$

The open set ${\mathfrak Q}_0$ (introduced at the end of section \ref{action}) lies in $G_1$, which lies in the projective space ${\mathbb P}\left(\bigwedge^{k+1}(U/C_1 \oplus U/C_2) \right)$ by the Pl\"ucker embedding. The (Zariski) closure of the image in ${\mathbb P}\left(\bigwedge^{\alpha_1}(U/C_1) \otimes \bigwedge^{\alpha_2} (U/C_2) \right)$ of ${\mathfrak Q}_0$ under $\pi$ is called {\em the variety ${\mathcal X}_{(\alpha_1, \alpha_2)}$ of generalized fundamental matrices or bifocal Grassmann tensors with profile} $(\alpha_1,\alpha_2)$. As proved in \cite{tubbAMPA}, it has dimension $(k+1)(h_1+h_2-k+1)-1$. Notice that the dimension does not depend on the profile. In fact, for each choice of $(\alpha_1, \alpha_2)$ such that $\alpha_1+ \alpha_2=k+1$, there exists a variety of bifocal Grassmann tensors ${\mathcal X}_{(\alpha_1, \alpha_2)}$. In other words, different profiles give different birational connected components. Moreover, as a consequence of \cite{Hart-Schaf}, a general point $p \in {\mathcal X}_{(\alpha_1, \alpha_2)}$ corresponds to a $\left(({\mathbb C}^*)^2/{\mathbb C}^*\right)$-orbit $[z \lambda M_1 | z \mu M_2]$ for $z, \lambda, \mu \in {\mathbb C}^*$. Every point of such an orbit corresponds to the generalized fundamental matrix $z^{k+1} \lambda^{\alpha_1} \mu^{\alpha_2} {\mathfrak F}$, where ${\mathfrak F}$ is the generalized fundamental matrix associated with $[M_1|M_2]$.  As a consequence, ${\mathcal X}_{(\alpha_1, \alpha_2)}$ can be viewed as a moduli space of $\left(({\mathbb C}^*)^2/{\mathbb C}^*\right)$-orbits of Grassmann tensors. 

\begin{theorem}
\label{homogeneous}
For each pair $(\alpha_1, \alpha_2)$ corresponding to a profile, the variety of bifocal Grassmann tensors ${\mathcal X}_{(\alpha_1, \alpha_2)}$ is birational to a homogeneous space with respect to the action of $GL(h_1+1) \times GL(h_2+1)$.
\end{theorem}
\begin{proof} As seen in the previous section, there is a right action of $GL(h_1 + 1) \times GL(h_2+1)$ on ${\mathfrak W}_0$, which induces an action on ${\mathfrak Q}_0$ . Notice that ${\mathfrak Q}_0$  is a homogeneous space with respect to the action of the group $GL(h_1+1) \times GL(h_2+1)$ because any matrix can be put in canonical form. This implies that for each pair $(\alpha_1, \alpha_2)$ there exists a (Zariski) non-empty open set in ${\mathcal X}_{(\alpha_1, \alpha_2)}$ that is a homogeneous space with respect to the action of $GL(h_1+1) \times GL(h_2+1)$.
\end{proof}

Now, set $s_j=h_j+1-\alpha_j$ and recall that $i=s_1+s_2+2$. Recall that by the Hodge operator we have
$$
\bigwedge^{\alpha_1} (U/C_1) \otimes \bigwedge^{\alpha_2} (U/C_2) \simeq \bigwedge^{s_1+1} (U/C_1)^{\vee} \otimes \bigwedge^{s_2+1} (U/C_2)^{\vee}.
$$
Therefore any $p \in {\mathcal X}_{(\alpha_1, \alpha_2)}$ corresponds to a $\left(({\mathbb C}^*)^2/{\mathbb C}^*\right)$-orbit of an $i$-dimensional subspace $T_p \subset (U/C_1)^{\vee} \oplus (U/C_2)^{\vee}$, i.e., a point in $G_2=G(i, (U/C_1)^{\vee} \oplus (U/C_2)^{\vee})$, which is mapped to ${\mathbb P}\left(\bigwedge^i\left((U/C_1)^{\vee} \oplus (U/C_2)^{\vee}\right)\right)$ under the Pl\"ucker embedding. As before, the linear projection $${\mathbb P}\left(\bigwedge^i\left((U/C_1)^{\vee} \oplus (U/C_2)^{\vee}\right)\right) \dashrightarrow {\mathbb P}\left(\bigwedge^{s_1+1} (U/C_1)^{\vee} \otimes \bigwedge^{s_2+1} (U/C_2)^{\vee} \right)$$ maps $G_2$ to a projective variety ${\mathcal X}_{(s_1, s_2)}$, which is birational to ${\mathcal X}_{(\alpha_1,\alpha_2)}$, as $G_2$ is the dual Grassmanniann of $G_1$. In what follows, we will focus our attention on ${\mathcal X}_{(s_1, s_2)}$; statements for ${\mathcal X}_{(\alpha_1,\alpha_2)}$ can be deduced in a similar fashion. Recall that $\dim({\mathcal X}_{(\alpha_1,\alpha_2)})= \dim({\mathcal X}_{(s_1, s_2)})=(k+1)(h_1+h_2-k+1)-1=(k+1)i-1$.

\begin{remark}
\label{strat}
By the decomposition described in \ref{effetau}, there exists a rational map $\varphi$ from the variety ${\mathcal X}_{(s_1, s_2)}$ to the secant variety $Sec_r(G(s_1+1, i) \times G(s_2+1,i))$. This map sends a general fundamental matrix ${\mathcal F}'$ in dual Pl\"ucker coordinates to the subspace generated by the $r$-tuple $\{(P_I, Q_{I^c})$ : $I$ \, \hbox{multi-index of length} $s_1+1$\} where $r$ is the rank of ${\mathcal F}'$. More precisely, there exists an isomorphism between $G(s_1+1,i)$ and $G(s_2+1,i)$ which is induced by the Hodge operator. Denote by $Graph(h) \subset G(s_1+1, i) \times G(s_2+1,i)$ the graph of this isomorphism, and by $Sec_r(Graph(h)) \subset Sec_r(G(s_1+1, i) \times G(s_2+1,i))$ the corresponding secant variety. Therefore, by the decomposition recalled before, the rational map $\varphi$ sends ${\mathcal X}_{(s_1, s_2)}$ to  $Sec_r(Graph(h))$. Since any linear combination of points $P_I$ yields a different fundamental matrix with the same image, the map $\varphi$ has a positive dimensional fiber. 

\end{remark}

\subsection{A natural action on ${\mathcal X}_{(s_1,s_2)}$} 

First, we investigate the action induced by \eqref{tauaction} on this variety of bifocal Grassmann tensors. As recalled before, any general point any $p \in {\mathcal X}_{(s_1, s_2)}$ corresponds to a $\left(({\mathbb C}^*)^2/{\mathbb C}^*\right)$-orbit of an $i$-dimensional subspace $T_p \subset (U/C_1)^{\vee} \oplus (U/C_2)^{\vee}$, which identifies a unique generalized fundamental form. The group $GL(h_1+1) \times GL(h_2+1) \times PGL(i)$ involved in \ref{tauaction} acts on $p$ by sending it to the point corresponding to the generalized fundamental matrix, which is identified by the $\left(({\mathbb C}^*)^2/{\mathbb C}^*\right)$-orbit $[(\Delta_1 \tau_1^T \Gamma^T )  |(\Delta_2 \tau_2^T\Gamma^T)]^T$.

\begin{lemma}
\label{effeaction}
Under the action $\ref{tauaction}$, and with the same notation adopted therein, a fundamental matrix $$
\mathfrak{F} = \frac{1}{det(G)}\sum_I \left(\left(\bigwedge^{s_1+1}K_1 \bigwedge ^{s_1+1}\tau_{1,c} \right)E_I \right) \otimes \left( \left(\bigwedge^{s_2+1}K_2\bigwedge^{s_2+1} \tau_{2,c}\right)F_{I^c}\right) $$ 
is sent to 
$$
\sum_I \left( \left(\bigwedge^{s_1+1} K_1 \bigwedge^{s_1+1} (\Gamma^T \tau_{1,c} \Delta_1^T) \right) E_I \right)\otimes \left(\left( \bigwedge^{s_2+1} K^T \bigwedge^{s_2+1} (\Gamma^T \tau_{2,c} \Delta_2^T)  \right) F_{I^c}\right)
$$
\end{lemma}
\begin{proof}
Let $[\tau_{1,c}|\tau_{2,c}]^T$ be the $(h_1 + h_2+2) \times i$ matrix defining the bifocal Grassmann tensor associated with the canonical form. Under the action in \eqref{tauaction}, this is mapped to $[(\Delta_1 \tau_{1,c}^T \Gamma^T )  |(\Delta_2 \tau_{2,c}^T\Gamma^T)]^T$. Corollary \ref{effetau} tells us how to associate the fundamental matrix with it.
\end{proof}

In particular, the right action of the group $GL(i)$ sends any generalized fundamental matrices to itself, as proved by the following result.
\begin{proposition} 
\label{contocri}
For any $\Gamma \in GL(i)$ one has 
$$
\left(\bigwedge^{s_1+1}\Gamma\right) \left(\sum_I E_I \otimes F_{I^c}\right)\left(\bigwedge^{s_2+1}\Gamma\right)^T = \det(\Gamma) \sum_I E_I \otimes F_{I^c}.
$$
\end{proposition}
\begin{proof}
Recall that, for $h=1,2$,  the elements of $\bigwedge^{s_h+1}\Gamma^T$ are the minors $m_{(i_1 \dots i_{s_h+1})}^{(j_1 \dots j_{s_h+1})}$ of the rows $(i_1 \dots i_{s_h+1})$ (with $i_1 < \dots <i_{s_h+1}$)  and of the columns $(j_1 \dots j_{s_h+1})$ (with $j_1 < \dots <j_{s_h+1}$) of $\Gamma^T$. The rows of $
(\bigwedge^{s_1+1}\Gamma^T)$ are indexed following the lexicographic order for the $s_1+1-$tuples $(i_1 \dots i_{s_1+1})$, while the columns of $
(\bigwedge^{s_1+1}\Gamma)$ are indexed following the lexicographic order for the $s_1+1-$tuples $(j_1 \dots j_{s_h+1}).$  Recall also that the only non vanishing entries of $\sum_I E_I \otimes F_{I^c}$ correspond to the $\pm 1$ on the secondary diagonal, as $\sum_I E_I \otimes F_{I^c}$ is the matrix associated with the Hodge operator. To prove the result, it suffices to apply the generalized Laplace expansion by complementary minors in order to see that the element of row $(i_1 \dots i_{s_2+1})$ and column  $(h_1 \dots h_{s_1})$ of the matrix at the left side of the equality in the statment is $\pm \det(\Gamma)$ if 
$(i_1 \dots i_{s_2+1})$ and  $(h_1 \dots h_{s_1})$ are complementary multi-indices, and zero otherwise. 

\end{proof}

Analogously to Lemma \ref{homogeneous}, the left action of $GL(h_1+1) \times GL(h_2+1)$ is transitive, so ${\mathcal X}_{(s_1,s_2)}$ is birational to a homogeneous space, which also follows from the birational equivalence with ${\mathcal X}_{(\alpha_1, \alpha_2)}$.


\subsection{A less natural action on ${\mathcal X}_{(s_1, s_2)}$} In what follows, we will prove a result on the geometric structure on ${\mathcal X}_{(s_1, s_2)}$. More precisely, pick a general point $p$ in ${\mathcal X}_{(s_1, s_2)}$. There exists a  $\left(({\mathbb C}^*)^2/{\mathbb C}^*\right)$-orbit of a subspace $[T_p \subset (U/C_1)^{\vee} \oplus (U/C_2)^{\vee}] \in G_2=G(i, h_1+h_2+2)$. Let us consider the group ${\mathcal H}$ of pairs $g = (\Delta_1, \Delta_2)$ where
$$
\Delta_1= \left(
\begin{array}{cc}
H & 0 \\
0 & V_1
\end{array}
\right), \qquad 
\Delta_2= \left(
\begin{array}{cc}
H & 0 \\
0 & V_2
\end{array}
\right),
$$
and $H \in GL(i)$ and $V_j \in GL(h_j+1-i)$ for $j=1,2$. The group $\left(({\mathbb C}^*)^2/{\mathbb C}^*\right)$ acts on ${\mathcal H}$ by sending $g=(\Delta_1, \Delta_2)$ to $(\zeta \alpha \Delta_1, \zeta \beta \Delta_2)$. Denote by ${\mathcal G}$ the quotient of ${\mathcal H}$ with respect to such an action, which has dimension $i^2+(h_1+1-i)^2+(h_2+1-1)^2-1$. The group ${\mathcal G}$  acts on ${\mathcal X}_{(s_1, s_2)}$ as follows. A general point $p$, which corresponds to a Grassmann tensor ${\mathfrak F}$ and is associated with the orbit $[z\lambda \tau_1| z \mu \tau_2]^T$, is sent to the point $q \in {\mathcal X}_{(s_1, s_2)}$ associated with $[z\zeta \lambda\alpha \tau_1\Delta_1^T|z \zeta \mu \beta \tau_2\Delta_2^T]^T$. The main result of this section is the following theorem, which will be proved in different steps.

\begin{theorem}
\label{ratfibr}
Let $\alpha_1, \alpha_2$ be a pair of non-negative integers such that $\alpha_1+\alpha_2=k+1$. Fix $h_1, h_2$ such that $k > \max \{h_1, h_2\}$ and $k \leq h_1+h_2+1$, as well as a $(k+1)$-dimensional vector space $U$. Set $s_j=h_j+1-\alpha_j$ for $j=1,2$. Then there exists a dominant rational map $\Psi: {\mathcal X}_{(s_1, s_2)} \dashrightarrow G(i, U^{\vee})$ such that the following hold:
\begin{itemize}
    \item $G(i,U^{\vee})$ is birationally ${\mathcal G}$-equivariant, that is, there exists a non-empty open set ${\mathfrak U}$ of ${\mathcal X}_{(s_1,s_2)}$ such that $\Psi(g.p)=\Psi(p)$ for every $p \in {\mathfrak U}$ and every $g \in {\mathcal G}$;
    \item the general orbit is isomorphic to $PGL(i)$.
\end{itemize}
\end{theorem}
\subsubsection{Step 1: the definition of $\Psi$} With the same notation adopted before, the inclusion of $T_p$ in $(U/C_1)^{\vee} \oplus (U/C_2)^{\vee}$ yields the horizontal short exact sequence in the diagram below. The map $j$ is given by $[\tau_1| \tau_2]^T$ after choosing suitable bases. Moreover, the matrix $[\tau_1| \tau_2]^T$ corresponds to a matrix $[M_1|M_2]$ where $C_j=ker(M_j^T)$. As a consequence, the map $\eta: U \rightarrow (U/C_1) \oplus (U/C_2)$ in \eqref{etadefinition}, which maps $u \in U$ to $M_1u-M_2u$, gives the vertical short exact sequence by duality; recall that $ker(\eta^{\vee})$ is isomorphic to $(U/C_1)^{\vee} \cap (U/C_2)^{\vee}$: see Lemma \ref{lemmanosnake}.
\begin{equation}
\label{diagrammacroce}
\xymatrix{
 & & 0 \ar[d] & &\\
 & & ker(\eta^{\vee}) \simeq (U/C_1)^{\vee} \cap (U/C_2)^{\vee} \ar[d]^{\gamma} & & \\
0 \ar[r] & T_p \ar[r]^j \ar@{-->}[dr] & (U/C_1)^{\vee} \oplus (U/C_2)^{\vee} \ar[r] \ar[d]^{\eta^{\vee}} & K \ar[r] & 0 \\
& & U^{\vee} \ar[d]& & \\
& & 0 & &
}
\end{equation}
By the Grassmann formula and the inequality $i < k+1$, the linear map $\eta^{\vee}$ generically maps $T_p$ to a subspace of $U^{\vee}$, which is isomorphic to $T_p$, as generically we have $\gamma(ker(\eta^{\vee})) \cap j(T_p)=\{0\}$. Then for a general point $p \in {\mathcal X}_{(s_1, s_2)}$ we set $\Psi(p)= [(\eta^{\vee} \circ j)(T_p) \subset U^{\vee}] \in G(i,U^{\vee})$.

\begin{remark}
In general, for any $i$-dimensional subspace $T_p$ the intersection $j(T_p) \cap Ker(\eta^{\vee})$ has dimension in $[0,i]$. When this dimension is $0$, we saw that $j(T_p)$ can be projected isomorphically onto $U^{\vee}$, as in Step 1. Therefore, the exceptional locus $Exc(\Psi)$ of $\Psi$ is given by the points $p$ such that the intersection $j(T_p) \cap Ker(\eta^{\vee})$ has dimension greater than or equal to $1$. If we denote by $Exc_j(\Psi)$ the points $p$ such that $\dim(j(T_p) \cap Ker(\eta^{\vee})) \leq j$ for $1 \leq j \leq i$, we have a stratification
$$
Exc_1(\Psi) \subseteq Exc_2(\Psi) \subseteq \ldots \subseteq Exc_i(\Psi).
$$
\end{remark}

\subsubsection{Step 2: the general fiber of $\Psi$} 
\begin{lemma}
\label{biequiv}
The map $\Psi$ is birationally ${\mathcal G}$-equivariant.
\end{lemma}
\begin{proof}
As claimed in Theorem \ref{ratfibr}, it suffices to prove that there exists a nonempty open set ${\mathfrak U}$ of ${\mathcal X}_{(s_1,s_2)}$ such that $\Psi(g.p)=\Psi(p)$ for every $p \in {\mathfrak U}$ and every $g \in {\mathcal G}$. Let ${\mathfrak U}$ be the maximal domain of definition of $\Psi$. By Step 1, any point $p \in {\mathfrak U}$ defines a vector space $T_p$ such that $j(T_p) \cap \gamma(ker(\eta^{\vee}))=\{0\}$. By definition of the group ${\mathcal G}$, and its elements $g$, the subspace $T_p$ is transformed by $H$ into another $i$-dimensional subspace $T_{g.p}$, which can not intersect $\gamma(ker(\eta^{\vee})$, as the transformation is bijective. Therefore, $T_p$ and $T_{g.p}$ are mapped one another by $H$. As a consequence, they define the same point in the Grassmanniann $G(i, U^{\vee})$. This can be summarized by saying that $\Psi(g.p)= \Psi(p)$ for a general point $p \in {\mathcal X}_{(s_1, s_2)}$ and every $g \in {\mathcal G}$.
\end{proof}

\begin{proposition}
\label{stabilizer}
The stabilizer of the action of ${\mathcal G}$ on ${\mathcal X}_{(s_1,s_2)}$ is given by the subgroup of matrices
$$
\Delta_1= \left(
\begin{array}{cc}
\delta I_i & 0 \\
0 & V_1
\end{array}
\right), \qquad 
\Delta_2= \left(
\begin{array}{cc}
\delta I_i & 0 \\
0 & V_2
\end{array}
\right),
$$
where $V_j \in GL(h_j+1-i)$ for $j=1,2$ and $\delta \in {\mathbb C}^*$. Therefore, the general fiber of $\Psi$ has dimension $i^2-1$.
\end{proposition}
\begin{proof}
Pick the point $p_c \in {\mathcal X}_{(s_1,s_2)}$ which corresponds to the $\left({({\mathbb C})^2}^*/{\mathbb C}^*\right)$-orbit of $T_{p_c}$ and the Grassmann tensor ${\mathfrak F}_c$. An element $g$ belongs to the stabilizer of $p_c$ if and only if $g.p_c=p_c$. The Grassmann tensor ${\mathcal F}_c$ is associated with the orbit $[z \lambda_1 \tau_{1,c}| z \lambda_2 \tau_{2,c}]^T$, where $\tau_{j,c}$ is given in \eqref{tauci} for $j=1,2$. The action $g.p_c$ is associated with the $\left({({\mathbb C})^2}^*/{\mathbb C}^*\right)$-orbit $[\zeta \alpha_1 \tau_{1,c}\Delta_1^T| \zeta \alpha_2 \tau_{2,c}\Delta_2^T]^T$. An element $g$ belongs to the stabilizer of $p_c$ if and only if $g.p_c=p_c$. Therefore for every $z$ and $\lambda_j$ we have
$$
\left(
\begin{array}{c}
z\lambda_j \\
0
\end{array}
\right)
= \tau_{j,c}^T=
\tau_{j,c}\Delta_j^T=
\left(
\begin{array}{c}
z\lambda_j \zeta \alpha_j H^T \\
0  
\end{array}
\right)
$$
This implies that $H=\frac{I_i}{\zeta \alpha_j}$. Hence the claim is proved because the dimension of the stabilizer is $$\dim(GL(h_1+1-i))+ \dim(GL(h_2+1-i))+\dim(Z(GL(i)) - \dim(({\mathbb C}^*)^2/({\mathbb C}^*))= $$
$$
=(h_1+1-i)^2+(h_2+1-i)^2,$$ where $Z(GL(i))$ is the group of scalar matrices in $GL(i)$.
\end{proof}

\begin{remark}
\label{finalone}
Let us consider the point $p_c \in {\mathcal X}_{(s_1, s_2)}$ that corresponds to the tensor ${\mathfrak F}_c$. The image $\Psi(p_c)$ is the $i$-dimensional subspace in $U^{\vee}$ generated by the rows of the matrix $(I_i | 0_{i,h_1+1-i}|0_{i,h_2+1-i})$ due to \eqref{tauci}. According to Proposition \ref{stabilizer}, the preimage of it with respect to $\Psi$ is a ${\mathcal G}$-orbit corresponding to Grassmann tensors of type $\bigwedge^{s_2+1}\Delta_1 {\mathfrak F_c} \bigwedge^{s_1+1}\Delta_1^T$, where
$$
\Delta_1= \left(
\begin{array}{cc}
H & 0 \\
0 & I_{h_1+1-i}
\end{array}
\right), \qquad 
\Delta_2= \left(
\begin{array}{cc}
H & 0 \\
0 & I_{h_2+1-i}
\end{array}
\right),
$$
and $H \in PGL(i)$.
\end{remark}

\subsubsection{Step 3: the map $\Psi$ is surjective}

\begin{corollary}
\label{dominance}
The rational map $\Psi: {\mathcal X}_{(s_1,s_2)} \dashrightarrow G(i, U^{\vee})$ is dominant.
\end{corollary}
\begin{proof}Let ${\mathcal I}$ be closure of $\Psi({\mathcal X}_{(s_1,s_2)})$. By Lemma \ref{stabilizer}, there is an orbit of maximal dimension (that of the point corresponding to ${\mathfrak F}_c$) which is $\dim({\mathcal G}) - \dim (Stab({\mathfrak F}_c))= i^2-1$. Therefore, by the Fiber Dimension Theorem, we have $$
\dim({\mathcal I}) \geq \dim\left({\mathcal X}_{(s_1, s_2)}\right) - i^2+1 =(k+1)(h_1+h_2-k+1)- i^2=\dim\left(G(i, U^{\vee})\right).
$$
Thus the claim follows.
Since a dominant map between projective varieties is surjective, every $i$-dimensional subspace in $U^{\vee}$ has a preimage under $\Psi$. 
\end{proof}

\begin{remark}
In other words, at least theoretically, given an $i$-dimensional space $W$ in $U^{\vee}$ it is possible to ``reconstruct a bifocal Grassmann tensor", i.e. a point in the preimage of $W$ under $\Psi$.
\end{remark}


\begin{thebibliography}{99}
\bibitem{oe1}
Chris Aholt and Luke Oeding.
\newblock The ideal of the trifocal variety.
\newblock {\em Math. Comp.}, 83(289):2553--2574, 2014.

\bibitem{al-to1}
Alberto Alzati and Alfonso Tortora.
\newblock A geometric approach to the trifocal tensor.
\newblock In {\em Journal of Mathematical Imaging and Vision}, 38 (2010) 159{170}.

\bibitem{al-to2}
Alberto Alzati and Alfonso Tortora.
\newblock Constraints for the trifocal tensor.
\newblock In {\em Tensors in image processing and computer vision}, Adv.
  Pattern Recognit., pages 261--269. Springer, London, 2009.


\bibitem{be-ber-no-tu}
Marina Bertolini, GianMario Besana, Roberto Notari, and Cristina Turrini.
\newblock Critical loci in computer vision and matrices dropping rank in codimension one.
\newblock {\em Journal of Pure and Applied Algebra}, ISSN 0022-4049. - 224:12(2020 May 26).


\bibitem{tubbLAIA}
Marina Bertolini, GianMario Besana, and Cristina Turrini.
\newblock Critical loci for projective reconstruction from multiple views in
  higher dimension: a comprehensive theoretical approach.
\newblock {\em Linear Algebra Appl.}, 469:335--363, 2015.

\bibitem{tubbAMPA}
Marina Bertolini, GianMario Besana, and Cristina Turrini.
\newblock  Generalized fundamental matrices as grassmann tensors.
\newblock {\em Annali di Matematica Pura ed Applicata (1923 -)}, 7 2016.

\bibitem{BBBT1}
Marina Bertolini, GianMario Besana, Gilberto Bini, and Cristina Turrini.
\newblock The rank of trifocal Grassmann tensors.
\newblock {\em SIAM J. Matrix Anal. Appl.}, ISSN 0895-4798. - 41:2(2020 Apr 29), pp. 591-604.

\bibitem{ber-tu-no1}
Marina Bertolini, Roberto Notari, and Cristina Turrini.
\newblock The {B}ordiga surface as critical locus for 3-view reconstructions.
\newblock {\em J. Symbolic Comput.}, 91:74--97, 2019.

\bibitem{be-tur1}
Marina Bertolini and Cristina Turrini.
\newblock Critical configurations for 1-view in projections from
  $\mathbb{P}^{k} \to \mathbb {P}^{2}$.
\newblock {\em Journal of Mathematical Imaging and Vision}, 27:277--287, 2007.


\bibitem{Hart-Zi2}
Richard Hartley and Andrew Zisserman.
\newblock {\em Multiple view geometry in computer vision}.
\newblock Cambridge University Press, Cambridge, second edition, 2003.
\newblock With a foreword by Olivier Faugeras.

\bibitem{Hart-Schaf}
Richard~I. Hartley and Frederik Schaffalitzky.
\newblock Reconstruction from projections using grassmann tensors.
\newblock {\em Int. J. Comput. Vision}, 83(3):274--293, July 2009.



\bibitem{H-K}
Richard Hartley and Fredrik Kahl. 
\newblock Critical Configurations for Projective Reconstruction from Multiple Views.
\newblock {\em Int J Comput Vision}, 71, 5–47 (2007). 
\newblock https://doi.org/10.1007/s11263-005-4796-1.

\bibitem{LA}
J.~M. Landsberg.
\newblock {\em Tensors: geometry and applications}, volume 128 of {\em Graduate
  Studies in Mathematics}.
\newblock American Mathematical Society, Providence, RI, 2012.

\bibitem{oe2}
Luke Oeding.
\newblock The quadrifocal variety.
\newblock {\em Linear Algebra Appl.}, 512:306--330, 2017.


\bibitem{L-M}
Max Lieblich  and Lucas Van Meter.
\newblock Two Hilbert Schemes in Computer Vision.
\newblock {\em SIAM J. Appl. Algebra Geometry}, 4(2), 297–321. 
\newblock https://doi.org/10.1137/18M1200117.


\bibitem{A-S-T}
Chris Aholt, Bernd Sturmfels and Rekha Thomas.
\newblock Hilbert Scheme in Computer Vision.
\newblock {\em Canadian Journal of Mathematics}, 65(5), 961-988. 
\newblock doi:10.4153/CJM-2012-023-2.

\bibitem{BigL}
Binglin Li.
\newblock Images of Rational Maps of Projective Spaces.
\newblock {\em International Mathematics Research Notices},Volume 2018, Issue 13, July 2018, Pages 4190–4228.
\newblock doi.org/10.1093/imrn/rnx003.


\bibitem{I-M-U}
Atsushi Ito, Makoto Miura and Kazushi Ueda.
\newblock Projective Reconstruction in Algebraic Vision.
\newblock {\em Canadian Mathematical Bulletin},63(3), 592-609.
\newblock doi:10.4153/S0008439519000687.

\end{thebibliography}
\end{document}